\documentclass{amsart}[12pt]

\usepackage[hidelinks]{hyperref}
\usepackage{yfonts} %
\usepackage{amssymb} %
\usepackage{amsthm}
\usepackage{verbatim}
\usepackage{array}
\usepackage{booktabs}%
\usepackage{hhline}%
\usepackage{xy} %
\usepackage{epsfig}%
\usepackage{color}%
\usepackage{upgreek}
\usepackage[english]{babel}
\usepackage{epigraph}%
\usepackage{fancybox}%
\usepackage{shadow}
\usepackage{afterpage}
\usepackage{mathrsfs}
\usepackage{enumitem}
\usepackage{tabularx}
\usepackage{subcaption}
\usepackage{graphicx}
\usepackage{type1cm}
\usepackage{eso-pic}
\usepackage{color}
\usepackage{upgreek}
\usepackage[foot]{amsaddr}

\newtheorem{theorem}{Theorem}
\newtheorem{lemma}{Lemma}
\newtheorem{proposition}{Proposition}
\theoremstyle{definition}

\newtheorem{example}{Example}

\theoremstyle{plain}
\newtheorem{corollary}{Corollary}

\newcommand{\otoprule}{\midrule[\heavyrulewidth]}
\newcommand{\vt}{\vspace{.1cm}}
\newcommand{\vtt}{\vspace{.2cm}}
\newcommand{\R}{\mathbb{R} }
\newcommand{\q}{\mathbb{Q} }

\newcommand{\h}{\mathbb{H}}

\newcommand{\hf}{\mathbb{H}_{\mathbb F}^m}

\newcommand{\s}{\mathbb{S}}

\renewcommand{\rho}{\varrho}

\renewcommand{\theta}{\varTheta}
\renewcommand{\Theta}{\varTheta}
\renewcommand{\Sigma}{\varSigma}
\renewcommand{\tau}{\uptau}
\usepackage{amsmath}

\newcommand{\overbar}[1]{\mkern 1.5mu\overline{\mkern-1.5mu#1\mkern-1.5mu}\mkern 1.5mu}

\makeatletter
\newcommand{\tpitchfork}{%
  \vbox{
    \baselineskip\z@skip
    \lineskip-.52ex
    \lineskiplimit\maxdimen
    \m@th
    \ialign{##\crcr\hidewidth\smash{$-$}\hidewidth\crcr$\pitchfork$\crcr}
  }%
}
\makeatother

\begin{document}

\title[]
{Weingarten Flows in Riemannian Manifolds}
\author{Ronaldo Freire de Lima}
\address[A1]{Departamento de Matem\'atica - Universidade Federal do Rio Grande do Norte}
\email{ronaldo.freire@ufrn.br}
\maketitle


\begin{abstract}
  Given orientable Riemannian manifolds $M^n$ and $\overbar M^{n+1},$
  we study  flows $F_t:M^n\rightarrow\overbar M^{n+1},$ called Weingarten flows,
  in which the hypersurfaces $F_t(M)$ evolve in the direction of their normal vectors
  with speed  given by a function $W$ of their principal curvatures,
  called a Weingarten function, which is
  homogeneous, monotonic increasing with respect to any of its variables, and
  positive on the positive cone. We obtain  existence results for flows
  with isoparametric initial data, in which the hypersurfaces
  $F_t:M^n\rightarrow\overbar M^{n+1}$  are all parallel, and $\overbar M^{n+1}$ is either
  a simply connected space form or a rank-one symmetric space of noncompact type.
  We prove that the avoidance principle holds for Weingarten flows defined by
  odd Weingarten functions, and also that such flows  are embedding preserving.
\end{abstract}

\section{Introduction}
Given an open set $\Gamma\subset\R^n$
containing $\Gamma_+:=\{(k_1,\dots, k_n)\,;\, k_i>0\},$
we say that
$W=W(k_1,\dots, k_n)\in C^\infty(\Gamma)$ is a \emph{Weingarten function} if it is
symmetric, homogeneous, monotonic increasing with respect to any of its variables, and
positive on $\Gamma_+.$
For a  hypersurface $f\colon M^n\rightarrow\overbar M^{n+1}$  ($M^n$ and $\overbar M^{n+1}$ are arbitrary
orientable Riemannian manifolds),
denote by $k_1, \dots ,k_n$  its principal curvature functions.
Assuming that $(k_1(p),\dots ,k_n(p))\in\Gamma$ for all $p\in M,$
we define the \emph{Weingarten function} $W_f$ of $f$ associated to $W$  as
\[
W_f(p):=W(k_1(p),\dots, k_n(p)), \,\,\,  p\in M.
\]
If $W_f$ is constant on $M,$ we say that $f$ is a \emph{$W$-hypersurface}.

The  higher order mean curvatures $H_r,$ $1\le r\le n,$
and the squared norm of the second fundamental form $\|A\|^2$
are  distinguished examples of   Weingarten functions. They are defined as
$$H_r=\sum_{i_1<\cdots <i_r}k_{i_1}\dots k_{i_r}\quad\text{and}\quad
\|A\|^2=\sum_{i=1}^{n}k_i^2.$$

In this paper, we shall consider the problem of finding a one-parameter family of smooth
oriented immersions $F(.\,,t):M^n\rightarrow\overbar M^{n+1},$ $t\in[0,T),$ which, for a given
Weingarten function $W\in C^\infty(\Gamma),$ satisfy
the evolution equation:
\begin{equation} \label{eq-Wflowintro}
\left\{
\begin{array}{l}
\displaystyle\frac{\partial F}{\partial t}(p,t)=W(p,t)N(p,t)\\[2ex]
F(p\,, 0)=f(p),
\end{array}
\right.
\end{equation}
where $N(p,t)$ is the
unit normal to the hypersurface $F_t:=F(.\,,t),$
and $W(.\,, t)=W_{F_t}$ is the Weingarten function of $F_t$  associated to $W.$
We shall call such a family of immersions a \emph{Weingarten flow}
(or  a $W$-\emph{flow}, in order to specify the function $W$) in $\overbar M^{n+1}$
with \emph{initial data} $f.$

Huisken and Polden \cite{huisken-polden} have established  existence of short time solutions
to \eqref{eq-Wflowintro}.
Here, we shall seek solutions  such that the immersions $F_t:M^n\rightarrow\overbar M^{n+1}$
are all parallel to the initial data $f,$ that is,
\begin{equation} \label{eq-parallelWflow}
F_t(p)=\exp_{f(p)}(\varphi(t)N(p)), \,\,\, (p,t)\in M\times [0,T),
\end{equation}
where $\exp$ stands for the exponential map of $\overbar M^{n+1},$ $\varphi\in C^\infty[0,T)$
satisfies $\varphi(0)=0,$ and  $N$ is the unit normal to $f.$ We call
$F_t$  a \emph{parallel $\varphi$-flow}, and
choose
\begin{equation} \label{eq-nt}
N(p,t)=d\exp_{f(p)}(\varphi(t)N(p))N(p)
\end{equation}
as the unit normal field of $F_t.$

As we shall see, if a parallel $\varphi$-flow is a solution
to \eqref{eq-Wflowintro}, then each  $F_t:M^n\rightarrow\overbar M^{n+1}$ is necessarily a
$W$-hypersurface.
This fact leads us to consider isoparametric hypersurfaces of $\overbar M^{n+1},$
i.e., those on which the principal curvatures  are constant functions.
In this context, the simply connected space forms $\q_\epsilon^{n+1}$ of constant
sectional curvature
$\epsilon\in\{0,1,-1\},$ as well as the rank-one symmetric
spaces of noncompact type (i.e., the hyperbolic spaces $\hf$), are natural sources
of parallel $W$-flows, since these spaces
have many families of isoparametric hypersurfaces
(see Section \ref{sec-parallelWflows} for details).

Our first main result, as stated below,  concerns parallel Weingarten flows in
space forms of non positive curvature.

\begin{theorem} \label{th-WflowsSpaceforms}
For $\epsilon\in\{0,-1\},$ let $f:M^n\rightarrow\q_\epsilon^{n+1}$
be a complete non totally geodesic isoparametric hypersurface of \,$\q_\epsilon^{n+1},$
and let $W\in C^\infty(\Gamma)$ be a Weingarten function.
Then,  there exists a parallel
$\varphi$-flow $F_t$ defined on a maximal interval $[0,T),$ $T\le +\infty,$
which is a solution to \eqref{eq-Wflowintro}, and has
the following properties, according to the isoparametric type of $f(M):$
\begin{itemize}[parsep=1ex]
  \item[\rm i)] If $f(M)\subset\h^{n+1}$ is a horosphere, then $T=+\infty$ and $\{F_t(M), t\in (0,+\infty)\}$ is a family of horospheres in
  $\h^{n+1}$ which foliates the open horoball bounded by $f(M).$
  \item[\rm ii)] If $f(M)\subset\h^{n+1}$ is an equidistant hypersurface to a
  totally geodesic hyperplane $\Pi\subset\h^{n+1},$
  then $T=+\infty$ and $F_t(M)\rightarrow \Pi$ as $t\rightarrow+\infty.$
  \item[\rm iii)] If $f(M)\subset\q_\epsilon^{n+1}$ is  either a geodesic sphere  or a generalized cylinder, then
  $T<+\infty,$  $\varphi(T)$ is the focal distance of $f(M),$  and $F_t(M)$ collapses
  into the focal set of $f(M)$ at $t=T.$
\end{itemize}
\end{theorem}

We also consider $W$-flows in $\s^{n+1}$ from its isoparametric hypersurfaces, obtaining
then the following result (see Section \ref{sec-parallelWflows} for definitions).

\begin{theorem} \label{th-WflowsSpheres}
Let $\mathscr F=\{f_\tau\colon M\rightarrow\s^{n+1}\,;\, \tau\in(0,\pi/g)\}$
be a family of positively oriented isoparametric  hypersurfaces of \,$\s^{n+1},$ and let
\,$W\in C^\infty(\Gamma)$ be  a Weingarten function such that
$W_{f_\tau}$ is well defined for all $f_\tau\in\mathscr F.$
Given $\tau_0\in(0,\pi/g),$ assume that
$W_{f_{\tau_0}}>0$ (resp. $W_{f_{\tau_0}}<0$), and that the function
$\tau\mapsto W_{f_{\tau_0-\tau}}$ (resp. $\tau\mapsto W_{f_{\tau_0+\tau}}$)
is increasing (resp. decreasing) on $[0,\tau_0)$  (resp. on $[0,\pi/g-\tau_0)$).
Under these conditions, the maximal  parallel $\varphi$-flow solution
$F_t=f_{\tau_0-\varphi(t)}$ to \eqref{eq-Wflowintro} with initial data $F_0=f_{\tau_0}$
collapses into the focal set $\mathscr F_+$ (resp. $\mathscr F_-$)
at  $t=\varphi^{-1}(\tau_0)$ (resp. $t=\varphi^{-1}(\tau_0-\pi/g)$).
\end{theorem}

It should be mentioned that Theorems \ref{th-WflowsSpaceforms} and \ref{th-WflowsSpheres}
constitute  extensions of the main results of \cite{dosreis-tenenblat}, where the authors
considered  mean curvature flows by parallel hypersurfaces in $\q_\epsilon^{n+1}.$

By considering isoparametric hypersurfaces of the hyperbolic spaces
$\hf,$ we obtain the following result. 

\begin{theorem} \label{th-WflowsHypspaces}
Let $f:M^n\rightarrow\hf$ be either a horosphere or
a geodesic sphere of \,$\hf,$ and let $W\in C^\infty(\Gamma)$ be a Weingarten function.
Then, there exists a parallel $\varphi$-flow $F_t$ defined on a maximal interval $[0,T),$ $T\le +\infty,$
which is a solution to  \eqref{eq-Wflowintro}.
In addition, the following  hold:
\begin{itemize}[parsep=1ex]
  \item[\rm i)] If $f(M)$ is a horosphere, then $T=+\infty$ and $\{F_t(M), t\in (0,+\infty)\}$ is a family of horospheres in
  $\hf$ which foliates the open horoball bounded by $f(M).$
  \item[\rm ii)] If $f(M)$ is  a geodesic sphere,  then
  $T<+\infty,$  $\varphi(T)$ is the radius of $f(M),$  and $F_t(M)$ collapses
  into the center of $f(M)$  at $t=T.$
\end{itemize}
\end{theorem}

An important property shared by many kinds of flows in Euclidean space is the \emph{avoidance principle},
which essentially says that two flows with disjoint initial data remain
disjoint until one of them colapses. Here, by means of a result by R. Hamilton, we establish an
avoidance principle for  $W$-flows $F_t:M^n\rightarrow\overbar M^{n+1}$ whose
Weingarten function $W\in C^\infty(\Gamma)$ is odd.
Setting $\mathbf{k}:=(k_1,\dots ,k_n),$
this means  that
$W$ admits an extension to $-\Gamma:=\{-\mathbf{k}\,;\, \mathbf{k}\in\Gamma\}$
which satisfies $W(-\mathbf{k})=-W(\mathbf{k}).$
As one can easily check, for $r$ odd, the mean curvatures $H_r$
are all odd functions.

\begin{theorem}[avoidance principle] \label{th-AP}

  Let $M_1^n,$ $M_2^n,$ and $\overbar M^{n+1}$ be  complete connected Riemannian manifolds, being $M_2^n$ compact.
  Assume that $W\in C^{\infty}(\Gamma)$  is an  odd Weingarten function,  and that
  $$F^i:M_i^n\times [0,T)\rightarrow\Omega\subset\overbar M^{n+1}, \,\,\, i=1,2,$$
  are $W$-flows, where $\Omega$ is a strongly convex open set of $\overbar M^{n+1}.$ 
  Under these conditions, we have that the function
  \[
  D(t):={\rm dist}^2(F_t^1(M_1),F_t^2(M_2)), \,\,\, t\in[0, T),
  \]
  is no decreasing. In particular, if
   $F_0^1(M_1)$ and $F_0^2(M_2)$ are disjoint, then
  $F_t^1(M_1)$ and $F_t^2(M_2)$ are disjoint for all $t\in [0,T).$
\end{theorem}

As a consequence of the avoidance principle,
if $\overbar M^{n+1}$ is either a space form $\q_\epsilon^{n+1}$  or
a hyperbolic space $\hf,$ then a  $W$-flow
$F_t:M^n\rightarrow\overbar M^{n+1}$ of a compact manifold
$M$ collapses in a finite time $T,$ provided that $W$ is odd, or
$F_t$ is an embedding for all $t\in[0,T)$ (see Corollary \ref{cor-finitetime} in Section \ref{sec-avoidance}).

In our final result, we show that  Weingarten flows
defined by odd Weingarten functions preserve embeddedness.

\begin{theorem} \label{th-embeddednesspreserving}
  Let  $\overbar M^{n+1}$ be a complete connected Riemannian manifold.
  Assume that $W\in C^{\infty}(\Gamma)$  is an  odd Weingarten function, and that
  $$F:M^n\times [0,T)\rightarrow\overbar M^{n+1}$$
  is a $W$-flow of a compact connected Riemannian manifold $M.$
  Under these conditions, if  the initial data $F_0$ is an embedding, then
  $F_t$ is an embedding for all $t\in [0,T).$
\end{theorem}

The paper is organized as follows. In Section \ref{sec-parallelflows}, we establish general facts on
$W$-flows by parallel hypersurfaces, and present the proofs of Theorems \ref{th-WflowsSpaceforms}--\ref{th-WflowsHypspaces}.
We also apply these results to  determine the collapsing time of some  $W$-flows in $\q_\epsilon^{n+1}$ and $\hf$  as well.
In Section \ref{sec-avoidance}, we provide the proofs of Theorems \ref{th-AP} and \ref{th-embeddednesspreserving}.

\section{$W$-Flows by Parallel Hypersurfaces}  \label{sec-parallelflows}

The following result gives us a way of
obtaining Weingarten flows  by parallel hypersurfaces.
An interesting property of such a flow
is that its hypersurfaces are all $W$-{hypersurfaces}.

\begin{proposition} \label{prop-solutionWflow}
Given a Weingarten function  $W\in C^\infty(\Gamma),$  let  $F_t$
be a parallel $\varphi$-flow  as in \eqref{eq-parallelWflow}, and assume
that $W_{F_t}$ is well defined for all $t\in [0,T).$
Then, $F_t$ is a solution to \eqref{eq-Wflowintro} with initial data $f=F_0$ if and only if
the function $\varphi$ satisfies
\begin{equation} \label{eq-ODE}
\varphi'(t)=W(p,t) \quad \forall (p,t)\in M\times[0,T).
\end{equation}
If so,  $F_t:M^n\rightarrow\overbar M^{n+1}$ is a $W$-hypersurface for all $t\in[0,T).$
\end{proposition}

\begin{proof}
From \eqref{eq-parallelWflow}, we have that
\[
\frac{\partial F}{\partial t}(p,t)=d\exp_p(\varphi(t)N(p))\varphi'(t)N(p)=\varphi'(t)N(p,t).
\]
This, together with \eqref{eq-nt}, gives that  $F_t=F(.\,, t)$ satisfies \eqref{eq-Wflowintro} if and only if
$\varphi$ satisfies \eqref{eq-ODE}.  In particular, if this
equality holds, the Weingarten function $W_{F_t}$ is constant
on $M$ (possibly depending on $t$), that is, $F_t$ is a $W$-hypersurface of $\overbar M.$
\end{proof}

As an immediate consequence of Proposition \ref{prop-solutionWflow}, we have:

\begin{corollary} \label{cor-existence}
Given  a Weingarten function $W\in C^\infty(\Gamma),$ let us suppose  that
\begin{equation}  \label{eq-parallelfamily}
\mathscr F:=\{f_\tau:M^n\rightarrow\overbar M^{n+1}\,;\, \tau\in (-\delta,\delta)\}
\end{equation}
is a family of parallel $W$-hypersurfaces of $\overbar M$ defined by
$f_\tau(p)=\exp_p(\tau N(p)),$ where
$N$ is the unit  normal to $f=f_0.$
Then, writing $W(\tau)=W_{f_\tau}$\,, we have that
the solution $\tau=\varphi(t)$ of the initial value problem
\begin{equation} \label{eq-ivp}
\left\{
\begin{array}{l}
\tau'=W(\tau) \\[1ex]
\tau(0)=0
\end{array}
\right.
\end{equation}
determines a parallel $\varphi$-flow solution to \eqref{eq-Wflowintro}.
\end{corollary}

As we pointed out in the introduction, the fact that   hypersurfaces of
parallel $W$-flows are $W$-hypersurfaces suggests the consideration of
isoparametric hypersurfaces.
Recall that a one-parameter family $f_\tau:M^n\rightarrow\overbar M^{n+1}$  of parallel hypersurfaces
is called \emph{isoparametric}
if, for each $\tau,$  any principal curvature function
$k_i$ of $f_\tau$ is constant on $M$ (possibly depending on $i$ and $\tau$).
In this case, each hypersurface
$f_\tau$ is also called \emph{isoparametric}.

Given a Weingarten function $W\in C^\infty(\Gamma),$ it is clear that
any isoparametric hypersurface $f_\tau:M^n\rightarrow\overbar M^{n+1}$
is a $W$-hypersurface, provided that  $W_f$ is well defined.
Therefore, in view of Corollary  \ref{cor-existence}, we have the following result.

\begin{corollary} \label{cor-isoparametric}
Suppose that  $f:M^n\rightarrow\overbar M^{n+1}$ is an isoparametric hypersurface.
Then, for any  Weingarten function $W\in C^\infty(\Gamma)$  for which $W_f$ is  well defined,
there exists a unique solution to \eqref{eq-Wflowintro}
by parallel hypersurfaces with initial data $f$.
\end{corollary}

\subsection{Parallel $W$-flows in space forms} \label{sec-parallelWflows}
Let us apply the results so far obtained to study $W$-flows in the simply connected
space forms $\q_\epsilon^{n+1}.$ In view of Corollary \ref{cor-isoparametric},  we
shall consider the isoparametric hypersurfaces of these spaces. (For details and proofs on this
subject we refer to \cite{cecil-ryan, dominguez-vazquez}.)

For $\epsilon\le 0,$ the complete isoparametric hypersurfaces of $\q_\epsilon^{n+1}$ are totally  classified.
They are:
\begin{itemize}[parsep=1ex]
  \item[i)] The totally geodesic hyperplanes $\q_\epsilon^{n}\subset \q_\epsilon^{n+1}.$
  \item[ii)] The geodesic spheres.
  \item[iii)] The generalized cylinders $\q_\epsilon^{n-k}\times\s^k,$ where $\q_\epsilon^{n-k}$ is a totally geodesic hypersurface
  of $\q_\epsilon^{n+1}$ of dimension $n-k<n,$  and $\s^k$ is the $k$-dimensional geodesic sphere of $\q_\epsilon^{n+1}.$
  \item[iv)] The horospheres of $\h^{n+1}.$
  \item[v)] The equidistant hypersurfaces to  totally geodesic hyperplanes of $\h^{n+1}.$
\end{itemize}

In fact, for $\epsilon\le 0,$ any isoparametric hypersurface of $\q_\epsilon^{n+1}$ is
necessarily an open set of one of the complete hypersurfaces listed above.

We point out that, in the cases (ii) and (iii), the  isoparametric hypersurfaces have focal points. More specifically,
any geodesic sphere has a unique focal point, which is its center, and the focal set of a generalized cylinder
$\q_\epsilon^{n-k}\times\s^k$ is the totally geodesic submanifold $\q_\epsilon^{n-k}.$ In such cases,
we shall take the focal distance as the parameter  for a family of isoparametric hypersurfaces, that is,
if $$\mathscr F=\{f_\tau:M^n\rightarrow\q_\epsilon^{n+1}; \tau\in I\subset\R\}$$ is such a family, then
$\tau$ is the distance from $f_\tau(M)$ to its focal set. For instance, if $\mathscr F$ is a family of concentric geodesic
spheres, then $M=\s^n$,  $I=[0,+\infty),$  and $\tau>0$ is the radius of $f_\tau(\s^n).$

We also observe that all of the above isoparametric hypersurfaces are connected, orientable,
properly embedded and, except for case (i), strictly convex.
We shall consider on them the orientation which makes all their principal curvatures positive.

\begin{proof}[Proof of Theorem \ref{th-WflowsSpaceforms}]
  Firstly, let us write $$\mathscr F=\{f_\tau\colon M^n\rightarrow\q_\epsilon^{n+1} \,;\, \tau\in I\subset\R\}$$
  for the isoparametric family of complete hypersurfaces of $\q_\epsilon^{n+1}$ (defined in a maximal interval
  $I\subset\R$) such that $f=f_{\tau_0},$ $\tau_0\in I.$
  From the convexity of the hypersurfaces $f_\tau$, and from the positivity  of
  $W$ on $\Gamma_+,$ we have that $W_{f_\tau}\ge 0$ for all $\tau\in I.$

  If $\mathscr F$ is a family of horospheres, then $I=\R$ and, for any $\tau\in\R,$ all principal curvatures of $f_\tau$ are
  equal to $1,$  which implies that $W_{f_\tau}=W$ is a positive constant
  independent of $\tau.$ Hence, by Corollary \ref{cor-existence},
  the function $\varphi(t)=Wt,$ $t\in[0,+\infty),$ determines a  $\varphi$-flow $F_t$  which is a solution
  to \eqref{eq-Wflowintro} with initial data $f_{\tau_0}.$ Namely, $$F_t=f_{Wt+\tau_0}, \, t\in [0,+\infty).$$
  Clearly, for all $t>0,$ $F_t(M)$ is a horosphere of $\h^{n+1}$ contained in the open horoball bounded by $f(M).$
  This proves (i).

  Assume now that $\mathscr F$ is a family of equidistant hypersurfaces to a
  totally geodesic hyperplane $\Pi\subset\h^{n+1}.$ In this case, $I=\R$ and the parameter
  $\tau>0$ is the distance from $f_\tau(M)$ to $\Pi.$ We can assume, without loss of generality, that
  $\tau_0> 0.$ Let $\varphi:[0,T)\rightarrow\R$ be the solution of \eqref{eq-ivp} defined
  in a maximal interval $[0,T).$ Then, $$F_t=f_{\tau_0-\varphi(t)}, \,\,\, t\in [0, T),$$ is a solution
  to \eqref{eq-Wflowintro} satisfying $F_0=f_{\tau_0}.$ Assume, by contradiction, that $T<+\infty.$ If
  $\varphi(T)=\tau_0,$  then the flow $F_t$ can be extended beyond $T$ just by setting $F_t=f_{0}$  for $t\ge T$
  (since, by the homogeneity of $W,$ $W(0,0,\dots,0)=0$), contradicting the maximality of $T.$ Analogously, if $\varphi(T)<\tau_0,$ we have that
  $F_T=f_{\tau_0-\varphi(T)}$ is well defined, so that we can extend the flow $F_t$ beyond $T$ --- again a contradiction.
  Therefore, $T=+\infty.$

  If $\varphi(t_0)=\tau_0$ for some $t_0\in (0,+\infty),$ then
  $F_t(M)=\Pi$ for all $t\ge t_0.$ Hence, we can assume that $\varphi$ is bounded
  above by $\tau_0.$ In this case, since $F_t(M)$ moves towards $\Pi,$ the principal curvatures of
  $F_t$ are positive decreasing  functions of $t.$ This, together with the
  monotonicity property of $W,$ gives that the function
  $W(\varphi(t))$ ($=W_{F_t}=W_{f_{\tau_0-\varphi(t)}}$)  decreases as $t\rightarrow+\infty.$
  Since $\varphi'(t)=W(\varphi(t)),$ we conclude  that $\varphi''(t)<0,$
  that is, $\varphi$ is positive, increasing, concave and bounded on $[0,+\infty).$
  These  properties clearly imply that
  $\varphi'(t)\rightarrow 0$ as $t\rightarrow+\infty.$ Therefore,
  \[
  W_{f_0}=0=\lim_{t\rightarrow+\infty}\varphi'(t)=\lim_{t\rightarrow+\infty}W(\varphi(t))=\lim_{t\rightarrow+\infty}W_{f_{\tau_0-\varphi(t)}},
  \]
  which yields $\lim_{t\rightarrow+\infty}\varphi(t)=\tau_0$\,. Consequently,
  $F_t\rightarrow f_0$ as $t\rightarrow+\infty,$ which shows assertion (ii).

  Finally, let us suppose that $\mathscr F$ is a family of concentric geodesic spheres
  of $\q_\epsilon^{n+1}$ (the argument for generalized cylinders is analogous). 
  In this setting, let $\varphi:[0,T)\rightarrow\R$ be the solution of \eqref{eq-ivp},  so that
  $F_t=f_{\tau_0-\varphi(t)}$ is the solution to \eqref{eq-Wflowintro} satisfying $F_0=f_{\tau_0}.$
  Since $F_t(M)$ flows towards the center of the spheres $f_\tau(M),$
  we have that $\varphi(t)<\tau_0$ for all $t\in[0, T),$  and also that $W(\varphi(t))=\varphi'(t)$ is a positive
  increasing function of $t.$ Thus, $\varphi''>0,$ that is, $\varphi$ is  bounded, increasing and strictly convex, which
  clearly implies that $T<\infty.$ Besides, we must have $\varphi(T)=\tau_0.$ Otherwise, arguing as in the preceding paragraph,
  we derive a contradiction by extending $\varphi$ beyond $T.$  This completes the proof of (iii), and so of the theorem.
\end{proof}

Let us consider now the isoparametric hypersurfaces of \,$\s^{n+1}.$
A well known result  asserts that any such hypersurface
has at most $g$ distinct principal curvatures, where
$g\in\{1,2,3,4, 6\}.$ The case $g=1,$ for instance, correspond to the geodesic spheres of
$\s^{n+1}.$ Rather than using the classification theorems for isoparametric
hypersurfaces of \,$\s^{n+1},$ we shall consider their
characterization as level sets of homogeneous polynomials, as  done by M\"unzner \cite{munzner}
(see also \cite{cecil-ryan}).

To be more clear, let  $f:M\rightarrow\s^{n+1}$ be an isoparametric hypersurface
with $g$ distinct principal curvatures.
M\"unzner's result asserts that
$f(M)$ is the intersection of $\s^{n+1}$ with a level set $P^{-1}(c),$ $c\in (-1,1),$
of a homogeneous polynomial function $P:\R^{n+2}\rightarrow\R$ of degree $g.$
Distinct level sets of $P$ are necessarily parallel in $\s^{n+1},$ and
the focal set of this parallel family has precisely two connected components, which are
the intersections of $\s^{n+1}$ with
$P^{-1}(-1)$ and $P^{-1}(1),$ respectively.
In addition, given $p\in M,$ if we write
$\gamma\colon (0,\pi/g)\rightarrow\s^{n+1}$   for the normalized geodesic from
$P^{-1}(1)$ to $P^{-1}(-1)$ which is orthogonal to $f$ at $p=\gamma(\tau),$ and set
the \emph{positive orientation}
$N(p)=-\gamma'(\tau)$ for $f,$
its $g$ distinct principal curvatures are given by
\begin{equation} \label{eq-principalcurvatures}
k_i=\cot\left(\tau+(i-1)\frac{\pi}{g}\right), \,\,\, 1\le i\le g.
\end{equation}
In this setting, $\tau\in (0,\pi/g)$ is the focal distance from $f(M)$ to $P^{-1}(1).$ Also,
all principal curvatures of $f$ increase as $\tau$ decreases to $0,$ and decrease as $\tau$ increases to
$\pi/g.$ We shall denote the multiplicity of $k_i$ by $m_i.$

Summarizing, we have  that any isoparametric hypersurface
of $\s^{n+1}$ with $g$ distinct principal curvatures is an element
of a family
$$\mathscr F=\{f_\tau\colon M\rightarrow\s^{n+1}\,; \tau\in(0,\pi/g)\}$$
of isoparametric hypersurfaces such that $f_\tau(M)$ is at a distance $\tau$
from the focal component $\mathscr F_+:=P^{-1}(1).$

\begin{proof}
[Proof of Theorem \ref{th-WflowsSpheres}]
Suppose that, for some $\tau_0\in(0,\pi/g),$
$W_{f_{\tau_0}}>0.$ In this case, if the function
$\tau\mapsto W_{f_{\tau_0-\tau}}$ is increasing on $[0,\tau_0),$
the $\varphi$-flow
$$F_t=f_{t-\varphi(t)}, \,\,\, \varphi(0)=0, \,\,\, \varphi'(t)=W_{f_{t-\varphi(t)}},$$
moves towards $\mathscr F_+$ with increasing velocity. Hence, arguing as in the proof of
Theorem \ref{th-WflowsSpaceforms}-(iii), we conclude that $F_t(M)$ collapses into
$\mathscr F_+$ at  $t=\varphi^{-1}(\tau_0).$

Analogously, if $W_{f_{\tau_0}}<0,$ and the function
$\tau\mapsto W_{f_{\tau_0+\tau}}$
is decreasing on the interval $[0,\pi/g-\tau_0),$  then $F_t(M)$ collapses into
the focal component
$\mathscr F_-:=P^{-1}(-1)$  at  $t=\varphi^{-1}(\tau_0-\pi/g).$
\end{proof}

Let us see now that Theorem \ref{th-WflowsSpheres} applies when $W$ is either
the higher order mean curvature $H_r$  or the squared norm
of the second fundamental form $\|A\|^2.$

Let $\mathscr F$ be as in Theorem \ref{th-WflowsSpheres}.
Then, in any open interval $(0,\tau_0),$ $0<\tau_0<\pi/g,$  we have that $k_1^\tau=\cot\tau$ is unbounded, whereas
$k_i^\tau=\cot(\tau+(i-1){\pi}/{g}), \, i=2,\dots, g,$ is bounded. Assuming that the multiplicity $m_1$ of
$k_1^\tau$ (which is the same for all $\tau$) satisfies $m_1\ge r,$ where $r\in\{1,\dots, n-1\},$
the $r$-th mean curvature $H_r(\tau)$ of $f_\tau$ is given by
\[
H_r(\tau)={{m_1}\choose{r}}\cot^r\tau+\sum_{i=0}^{r-1}\mu_i(\tau)\cot^i\tau,
\]
where the functions $\mu_i$ are all bounded in $(0,\tau_0).$ In particular, if $\tau_0$ is sufficiently
small, $H_r(\tau_0)>0,$ and the function $\tau\in [0,\tau_0)\mapsto H_r(\tau_0-\tau)$ is increasing.
In the same manner, if $r$ is odd, $m_g\ge r,$  and $\tau_0$ is sufficiently close to $\pi/g,$ then
$H_r(\tau_0)<0,$ and the function $\tau\in [0,\pi/g-\tau_0)\mapsto H_r(\tau_0+\tau)$ is decreasing.
Thus, we have the following

\begin{corollary}
Theorem \ref{th-WflowsSpheres} applies to the Weingarten function $W=H_r,$
$1\le r\le n-1.$ More precisely, given $\tau_0\in(0,\pi/g),$ if $m_1\ge r$ (resp. $m_g\ge r,$  $r$  odd),
and $H_r(\tau_0)>0$ (resp. $H_r(\tau_0)<0$),  the maximal  parallel $\varphi$-flow solution
$F_t=f_{\tau_0-\varphi(t)}$ to  $H_r$-flow with initial data $F_0=f_{\tau_0}$
collapses into the focal set $\mathscr F_+$ (resp. $\mathscr F_-$)
at  $t=\varphi^{-1}(\tau_0)$ (resp. $t=\varphi^{-1}(\tau_0-\pi/g)$).
\end{corollary}

Theorem \ref{th-WflowsSpheres} also applies to the norm of the second fundamental form
$\|A\|^2,$ since $\|A\|^2(\tau_0-\tau)$ is clearly increasing on $[0,\tau_0)$ for all
sufficiently small $\tau_0\in(0,\pi/g).$

\begin{corollary}
Let $\mathscr F$ be as in Theorem \ref{th-WflowsSpheres}.  Given
a sufficiently small $\tau_0\in(0,\pi/g),$
the maximal  parallel $\varphi$-flow solution
$F_t=f_{\tau_0-\varphi(t)}$ to  $\|A\|^2$-flow with initial data $F_0=f_{\tau_0}$
collapses into the focal set $\mathscr F_+$
at  $t=\varphi^{-1}(\tau_0).$
\end{corollary}

Next, we apply the results of this section to determine the collapsing time of
some parallel $W$-flows in $\q_\epsilon^{n+1}.$
For that, we shall consider the trigonometric functions $\cos_\epsilon$ and $\sin_\epsilon$
as defined  in Table \ref{table-trigfunctions}. The functions $\tan_\epsilon,$
$\cot_\epsilon,$ and $\sec_\epsilon$ are defined accordingly, that is,
$\tan_\epsilon={\sin_\epsilon}/{\cos_\epsilon},$
$\cot_\epsilon={\cos_\epsilon}/{\sin_\epsilon},$ and $\sec_\epsilon=1/\cos_\epsilon.$
\begin{table}[thb]%
\centering %
\begin{tabular}{cccc}
\toprule %
   {{\small\rm Function}}                & $\epsilon=0$ & $\epsilon=1$   & $\epsilon=-1$ \\\otoprule %
$\cos_\epsilon (s)$    & $1$          & $\cos s$        & $\cosh s$     \\\midrule
$\sin_\epsilon (s)$    & $s$          & $\sin s$         & $\sinh s$   \\\bottomrule
\end{tabular}
\vtt
\caption{\small Definition of $\cos_\epsilon$ and $\sin_\epsilon$}
\label{table-trigfunctions}
\vspace{-.8cm}
\end{table}

\begin{example}[\emph{parallel $\|A\|^2$-flow in $\q_\epsilon^{n+1}$ with spherical initial data}]
Let $$f\colon\s^n\rightarrow\q_\epsilon^{n+1}$$ be a
(totally umbilical) strictly convex geodesic sphere
of $\q_\epsilon^{n+1}$ of radius $R>0$ and principal curvature $k=\cot_\epsilon R.$
Set
\[
\mathcal R_\epsilon=:
\left\{
\begin{array}{ccl}
  \pi/2 & \text{if} & \epsilon=1 \\
  +\infty & \text{if} & \epsilon\ne 1
\end{array}
\right.
\]
and let $$\mathscr F=\{f_\tau:\s^n\rightarrow\q_\epsilon^{n+1}\,;\, \tau\in (0,\mathcal R_\epsilon)\}$$
be the family of parallel geodesic spheres of $\q_\epsilon^{n+1}$ such that $f_R=f.$
By Theorems \ref{th-WflowsSpaceforms} and \ref{th-WflowsSpheres}, for $W=\|A\|^2,$ the flow
  \[
  F_t:=f_{R-\varphi (t)}, \,\,\, t\in[0,T),
  \]
  where $\varphi$ satisfies
  \begin{equation} \label{eq-ODE2}
  \varphi'(t)=W_{f_{R-\varphi(t)}}=n\cot_\epsilon^2(R-\varphi(t)), \,\,\, \varphi(0)=0,
  \end{equation}
  is a solution to \eqref{eq-Wflowintro} which collapses into the center of $f(\s^n)$ at time $T=\varphi^{-1}(R).$
  Separating variables in \eqref{eq-ODE2}, we obtain the equation
  $$\tan_\epsilon^2(R-\varphi)d\varphi=ndt,$$
  which yields
  \begin{equation}\label{eq-varphi}
  \begin{aligned}
    (R-\varphi(t))^3&=R^3-3nt \quad (\text{for} \,\, \epsilon=0),\\[1ex]
    \tan_\epsilon(R-\varphi(t))+\varphi(t)&=\frac{1}{k}-\epsilon nt \quad\,\,\,(\text{for} \,\, \epsilon=\pm1).
  \end{aligned}
\end{equation}

  Hence, by making  $t=T=\varphi^{-1}(R)$ in \eqref{eq-varphi}, one concludes that the collapsing time $T$
  for the $\|A\|^2$-flow $F_t$ with initial data $f$ is 
\begin{equation} \label{eq-T}
T=
\left\{
\begin{array}{cl}
  \frac{R^3}{3n} & (\text{for} \,\, \epsilon =0), \\[1.5ex]
  \frac{\epsilon(1-kR)}{kn} & (\text{for} \,\, \epsilon =\pm 1).
\end{array}
\right.
\end{equation}
\end{example}

\begin{example}[\emph{parallel \,$H_r$-flow in \,$\q_\epsilon^{n+1}$ with spherical initial data}]
Let $f$ and  $\mathscr F$
be as in the preceding example. For the $H_r$-flow, the differential equation for $\varphi$ is
\begin{equation} \label{eq-ODE3}
\varphi'(t)=W_{f_{R-\varphi(t)}}={n\choose r}\cot_\epsilon^r(R-\varphi(t)), \,\,\, \varphi(0)=0,
\end{equation}
which separates as
\begin{equation} \label{eq-ode4}
\tan_\epsilon^r(R-\varphi)d\varphi={n\choose r}dt.
\end{equation}

For $\epsilon=0,$ the solution $\varphi$ is given implicitly  by
\[
\frac{(R-\varphi(t))^{r+1}}{r+1}=\frac{R^{r+1}}{r+1}-{n\choose r}t,
\]
which yields
\[
T={n\choose r}^{-1}\frac{R^{r+1}}{r+1}
\]
for the collapsing time of $F_t.$

For $\epsilon=\pm1,$ integration on the left hand side of \eqref{eq-ode4} is recurrent. In
Table \ref{table-T}, we list the solutions $\varphi$ and corresponding collapsing times for $r=1,2.$

\begin{table}[thb]%
\centering %
\begin{tabular}{ccccc}
\toprule %
  $r$   &            & $\varphi$ &  &$T$      \\\otoprule %
   1     &            & ${\cos_\epsilon(R-\varphi(t))}=e^{\epsilon nt}{\cos_\epsilon R}$& &$\frac{\epsilon}{n}\log(1/\cos_\epsilon R)$ \\\midrule
   2     &            & $\tan_\epsilon(R-\varphi(t))+\varphi(t)=\frac{1}{k}-\epsilon\frac{n(n-1)}{2}t$        & & $\frac{2\epsilon(1-kR)}{kn(n-1)}$     \\\bottomrule
\end{tabular}
\vtt
\caption{\small Function $\varphi$ and collapsing time $T$ for spherical parallel $H_r$-flows.}
\label{table-T}
\vspace{-.6cm}
\end{table}
\end{example}


\begin{example}[\emph{parallel $K$-flow  in \,$\s^{n+1}$ with non spherical initial data}]
Consider an isoparametric family $\mathscr F$ of hypersurfaces
$f_\tau:M^n\rightarrow\s^{n+1},$ $\tau\in (0,\pi/2),$
with two distinct principal curvatures
\[
k_1^\tau=\cot\tau \quad\text{and}\quad k_2^\tau=\cot(\tau+\pi/2)=-\tan\tau,
\]
whose  multiplicities are $m_1$ and $m_2,$ respectively. By a result due to Cartan,
$M$ is homeomorphic to the product $\s^{m_1}\times\s^{m_2},$ and the focal components
$\mathscr F_-$ and $\mathscr F_+$ are isometric to the standard spheres
$\s^{m_1}$ and $\s^{m_2},$ respectively.
Assuming
$m_2$ even, we have that the Gaussian curvature $K(\tau)$ of $f_\tau$ is
\[
K(\tau)=\cot^{m_1}(\tau)\tan^{m_2}(\tau),
\]
which is clearly a positive function on $(0,\pi/2).$

If $m_1=m_2,$ then $K=1$ for all $\tau\in(0,\pi/2).$ In this case, given $\tau_0\in (0,\pi/2),$
the flow $F_t=f_{\tau_0-t}$ is a solution to $K$-flow with initial data $f_{\tau_0}$ and collapsing
time $T=\tau_0.$

If $m_1>m_2,$ the function $K(\tau_0-\tau)=\cot^{m_1-m_2}(\tau_0-\tau)$ is increasing
in $[0,\tau_0).$ So, considering the solution $\varphi$ of  $\tau'=K(\tau)$ such that $\varphi(0)=0,$
we have from Theorem \ref{th-WflowsSpheres} that the flow $F_t=f_{\tau_0-\varphi(t)}$ collapses
into $\mathscr F_+$ at $T=\varphi^{-1}(\tau_0).$ Besides, setting $m=m_1-m_2,$ we obtain
the function $\varphi$ and the collapsing time $T$ by integrating $\tan^m(\tau_0-\varphi)$ with respect
to $\varphi,$ as in the preceding example.
For instance, if $m=1,$  the implicit equation for $\varphi$ and collapsing time
$T$ are
\[
{\cos(\tau_0-\varphi(t))}=e^t{\cos\tau_0} \quad\text{and}\quad T=\log(\sec\tau_0).
\]

An analogous reasoning applies if $m_2$ is odd and $m_1<m_2,$ in which case
$F_t$ collapses into $\mathscr F_-$ at  $T=\varphi^{-1}(\tau_0-\pi/2).$
\end{example}


\subsection{Parallel $W$-flows in rank-one symmetric spaces}
Let us consider now the rank-one symmetric spaces of non com\-pact type,
which are precisely the hyperbolic spaces
described through  the four normed division algebras:
$\R$ (real numbers), $\mathbb C$ (complex numbers),
$\mathbb K$ (quaternions) and $\mathbb O$ (octonions). They are denoted by
$\h_\R^m, \,\,\h_\mathbb C^m, \,\, \h_{\mathbb K}^m$ {and} $ \h_{\mathbb O}^2$
and called \emph{real hyperbolic space, complex hyperbolic space, quaternionic hyperbolic space} and
\emph{Cayley hyperbolic plane,} respectively. The real hyperbolic space is $\h^{n+1},$ i.e., the simply connected
space form of constant sectional curvature $-1.$
We shall adopt the unified notation $\hf$ for these hyperbolic spaces, where $m=2$ for $\mathbb F=\mathbb O.$
The real dimension  of $\hf$ is $n+1=m\dim \mathbb F.$ In particular,
$\h_{\mathbb O}^2$ has dimension $n+1=16.$

We add that any hyperbolic space $\hf$ is a Hadamard manifold with
negative bounded sectional curvature everywhere and, more importantly,
their geodesic spheres and horospheres are all isoparametric and strictly convex
(see \cite{berndtetal,dominguez-vazquez} for details and proofs).

\begin{proof}[Proof of Theorem \ref{th-WflowsHypspaces}]
Suppose that $f(M)$ is a geodesic sphere and set
\begin{equation} \label{eq-spheresHF}
\mathscr F=\{f_\tau:\s^n\rightarrow\hf\,;\,\tau\in(0,+\infty)\}
\end{equation}
for  the isoparametric family of  geodesic spheres of $\hf$
such that $f=f_{\tau_0}$ for some $\tau_0\in (0,+\infty).$
(Recall that the parameter $\tau$ is the radius of $f_\tau(\s^n).$)

The principal curvatures $k_i^\tau$ of $f_\tau$ with respect to the inward orientation are
\begin{equation} \label{eq-princcurvhf}
\begin{aligned}
    k_1^\tau &= \coth(\tau) \,\,\, \text{with multiplicity} \,\,\,  q\\[1.5ex]
    k_2^\tau &= \frac{1}{2}\coth(\tau/2) \,\,\,\text{with multiplicity}\,\,\,  n-q,
  \end{aligned}
\end{equation}
where $q=n$ for $\h_\R^{n+1}$,  $q=1$ for $\h_\mathbb C^m$, $q=3$ for $\h_{\mathbb K}^m$, and
$q=7$ for $\h_{\mathbb O}^2$  (see, e.g., \cite[pgs. 353, 543]{cecil-ryan} and \cite{kimetal}).
In particular, we have
\[
\lim_{\tau\rightarrow 0}k_i^\tau=+\infty, \,\,\, i=1,2.
\]

From the above considerations (and the monotonicity property of $W$),
just as in the real case, we conclude that  the parallel $\varphi$-flow
with initial data $f=f_{\tau_0}$
collapses to its center at $T=\varphi^{-1}(\tau_0).$

As we pointed out, the horospheres of any hyperbolic space $\hf$ are  isoparametric. In fact,
as in the real case, they foliate $\hf$ and have all the same constant principal curvatures
(cf. the proposition on page 88 of \cite{berndtetal}).
So, any horosphere of $\hf$ moves indefinitely with constant speed under any $W$-flow.
\end{proof}

In the next example, we calculate the collapsing time of a geodesic sphere of
$\hf$ moving under $H$-flow.

\begin{example}[\emph{parallel $H$-flow in $\hf$ with spherical initial data}]
Let $\mathscr F$ be as in \eqref{eq-spheresHF}. Then, the mean curvature $H(\tau)$ of
$f_\tau$ is
\[
H(\tau)=q\coth(\tau)+\frac{n-q}{2}\coth(\tau/2).
\]

Given $R\in (0,+\infty),$
by Corollary  \ref{cor-existence} and Theorem \ref{th-WflowsHypspaces},  the flow
\[
F_t:=f_{R-\varphi (t)}, \,\,\, t\in[0,T),
\]
where $\varphi$ satisfies
\begin{equation} \label{eq-ODE20}
\left\{
\begin{array}{l}
\varphi'(t)=H(R-\varphi(t))=q\coth(R-\varphi(t))+\frac{n-q}{2}\coth((R-\varphi(t))/2)\\ [1ex]
\varphi(0)=0,
\end{array}
\right.
\end{equation}
is a solution to \eqref{eq-Wflowintro} with initial data $f_R$ which
collapses into the center of $f_R(\s^n)$ at time $T=\varphi^{-1}(R).$

From  \eqref{eq-ODE20}, we obtain the equation
$$\frac{d\varphi}{q\coth(R-\varphi)+((n-q)/2)\coth((R-\varphi)/2)}=dt.$$
Setting $x=e^{R-\varphi}$ and integrating the resulting rational function $f(x)/g(x)$
by means of the identities
\begin{itemize}[parsep=1ex]
  \item $\int {\frac  {xdx}{ax^{2}+bx+c}}={\frac  {1}{2a}}\log \left|ax^{2}+bx+c\right|-{\frac  {b}{2a}}\int {\frac  {dx}{ax^{2}+bx+c}}+C,$
  \item $\int {\frac  {dx}{x(ax^{2}+bx+c)}}={\frac  {1}{2c}}\log \left|{\frac  {x^{2}}{ax^{2}+bx+c}}\right|-{\frac  {b}{2c}}\int {\frac  {dx}{ax^{2}+bx+c}}+C,$
\end{itemize}
we conclude that the solution $\varphi$ of  \eqref{eq-ODE20} is given implicitly by
\[
\log\left(\frac{e^{R-\varphi(t)}}{a(e^{2(R-\varphi(t))}+1)+be^{R-\varphi(t)}}\right)^{1/a}=t+C(R),
\]
where $a=({n+q})/{2},$ \,$b=n-q,$ and
\[
C(R)=\log\left(\frac{e^{R}}{a(e^{2R}+1)+be^{R}}\right)^{1/a}.
\]
Therefore, the collapsing time $T=\varphi^{-1}(R)$ is
\[
T=\log\left(\frac{a(e^{2R}+1)+be^R}{2ne^R}\right)^{\frac{1}{a}}.
\]
\end{example}


\section{Avoidance Principle for Weingarten Flows} \label{sec-avoidance}

In this section, we prove Theorem \ref{th-AP}, which constitutes an avoidance principle
for  Weingarten flows whose corresponding Weingarten functions  are  odd,  as we mentioned in the introduction.
The fundamental property of such a $W$-flow $F_t\colon M^n\rightarrow\overbar M^{n+1}$ is that it
is invariant under change of orientation.
Indeed, given  $(p,t)\in M\times [0,T),$
writing $k_i=k_i(p,t)$ and $N=N(p,t),$ one has
\begin{eqnarray}
W(-k_1,\dots ,-k_n)(-N)&=& -W(k_1,\dots ,k_n)(-N)\nonumber\\
                                      &=& W(k_1,\dots ,k_n)N\label{eq-invariantchangeorientation}\\
                                      &=& \frac{\partial F}{\partial t}(p,t).\nonumber
\end{eqnarray}

Along the proof of Theorem \ref{th-AP}, we shall  consider  graphs over
tangent spaces of hypersurfaces, as described below.

Let  $f:M^n\rightarrow\overbar M^{n+1}$ be an oriented hypersurface. Fix
$p\in M,$ and let $U\subset T_pM$ be an open neighborhood of the zero
vector of the tangent space of $M$ at $p.$ Given a  function
$\phi\in C^{\infty}(U),$ we call the set (assuming it is well defined)
\[
\Sigma_\phi:=\{\exp_{f(p)}(v+\phi(v)N(p))\in\overbar M\,;\, v\in U\}
\]
the \emph{graph} of $\phi$ on $U.$ Here, $\exp$ denotes the exponential map
of $\overbar M^{n+1}.$

Clearly, $\Sigma_\phi$ is an orientable hypersurface of $\overbar M^{n+1}.$ Moreover,
it is a well known fact that, if the zero vector $0\in U$ is a critical point
of $\phi,$ then the Hessian of $\phi$ at $0$ coincides with the second fundamental form
of $\Sigma_\phi$ at $\overbar p=\exp_{f(p)}(\phi(0)N(p))\in\Sigma_\phi.$ In this case,
we consider in  $\Sigma_\phi$ the orientation
such that the unit normal to $\Sigma_\phi$ at $\bar p$ is (cf. Theorem 3 in \cite{bishop-crittenden}, pg. 198):
$$N_\phi(\bar p)=d\exp_{f(p)}(\phi(0)N(p))N(p).$$
Notice that, if $\gamma:[0,L]\rightarrow\overbar M^{n+1}$ is the normalized geodesic from
$f(p)$ to $\bar p$ satisfying $\gamma'(0)=N(p),$ then $N_\phi(\bar p)=\gamma'(L).$ 


The following elementary result, which  will be useful to us,
compares principal curvatures of graphs whose corresponding functions have
a common critical point.
We adopt the convention of ordering
the principal curvatures  as $k_1\le \cdots \le k_n$\,.

\begin{lemma} \label{lem-graph}
With the above notation, assume that $\phi, \mu\in C^{\infty}(U)$
satisfy $\mu\ge\phi$ on $U,$ and that the null vector $0\in U$ is a minimum
of $\mu-\phi.$ Then, any principal curvature of $\Sigma_\mu$ at
$\overbar p =\exp_{f(p)}(\mu(0)N(p))$ is greater than, or
equal to, the corresponding principal curvature of $\Sigma_\phi$ at
$\overbar q =\exp_{f(p)}(\phi(0)N(p))$.
\end{lemma}

\begin{proof}
Since $0$ is a minimum of $\mu-\phi,$ we have that the Hessian of
$\mu-\phi$ at $0$ is positive semi-definite, which implies that the same is true
for the operator $A_\mu-A_\phi$\,, where $A_\mu$ and $A_\phi$ are the shape operators
of $\Sigma_\mu$ at $\bar q$ and $\Sigma_\phi$ at $\bar p$, respectively. However,
a standard result in Linear Algebra (see theorem on page 130 in \cite{gelfand})
asserts the following: If $A$ is  self-adjoint
and $B$ is positive semi-definite, then the eigenvalues of $A$ do not exceed
the corresponding ones of $A+B.$ Hence, setting $A=A_\phi$ and $B=A_\mu-A_\phi$\,,
the lemma follows.
\end{proof}

The next result, due to R. Hamilton \cite{hamilton} (see also \cite{mantegazza}),
will play a fundamental role in the sequel.

\begin{lemma}[Hamilton's trick] \label{lem-HT}
Let $u\colon M\times[0,T)\rightarrow\R$ be  a $C^1$ function
with the following property: For  each $t_0\in [0,T),$  there exist $\delta>0$ and a compact subset
$\Omega\subset M-\partial M$ such that, for any $t\in(t_0-\delta,t_0+\delta),$ the minimum
\[
u_{\rm min}(t):=\min_{p\in M}u(p,t)
\]
is attained (at least) at one point of $\Omega$.
Then, the function $u_{\rm min}$ is locally Lipschitz in $(0,T)$ and,
for each $t\in (0,T)$ where it is differentiable, one has
\[
u_{\rm min}'(t)=\frac{\partial u}{\partial t}(p_0,t),
\]
where  $p_0\in M-\partial M$ is any interior point at which $u(.\,, t)$
attains its minimum.
\end{lemma}

\begin{proof}[Proof of Theorem \ref{th-AP}]

Since  $M_2$ is compact, for each $t\in (0,T),$ there exists
a pair $(p_1,p_2)\in M_1\times M_2$ (possibly depending on $t$) such that
\begin{equation} \label{eq-1}
D(t)={\rm dist}^2(F_t^1(p_1),F_t^2(p_2)).
\end{equation}
In addition, ${\rm dist}^2$ is  smooth on $\Omega,$ which implies that the function
\[
u(p,q,t):={\rm dist}^2(F_t^1(p),F_t^2(q)), \,\,(p,q,t)\in M_1\times M_2\times [0,T),
\]
is smooth as well.  Thus,  Hamilton's Trick applies and gives that
$D(t)=u_{\min}(t) $ is locally Lipschitz, so that $D$
is differentiable almost everywhere (by  Rademacher Theorem).
Also, at a differentiable point $t_0,$ the following equality holds
\begin{equation} \label{eq-2}
D'(t_0)=\frac{\partial u}{\partial t}(p_1,p_2,t_0),
\end{equation}
where $(p_1,p_2)\in M_1\times M_2$ is any pair at which $u(.\,, t_0)$ attains its minimum.
So, it suffices to prove that $D'(t_0)\ge 0.$ This is certainly true if $D(t_0)=0$
(since $D$ is nonnegative), so that we can assume $D(t_0)\ne 0.$

In the above setting, the minimizing normalised
geodesic $\gamma_{t_0}\colon [0,L]\rightarrow\overbar M$
joining the points $\bar p_1:=F_{t_0}^1(p_1)$ and $\bar p_2:=F_{t_0}^2(p_2)$ is orthogonal
to both $F_{t_0}^1(M_1)$ (at $\bar p_1=\gamma_{t_0}(0)$) and $F_{t_0}^2(M_2)$  (at  $\bar p_2=\gamma_{t_0}(L)$).


Let us denote by $\Pi$ the tangent space of $F_{t_0}^1(M_1)$ at $\bar p_1$\,.
It is easily checked that, for $i=1,2,$
there exists an open neighborhood  $U$ of $0$  in $\Pi$
such that, for all $t$ sufficiently close to $t_0,$ in a suitable
neighborhood of $\bar p_i$  in $\overbar M,$
$F_{t}^i(M_i)$  is a
graph of a function $\phi_t^i\in C^{\infty}(U).$
In particular, we have $\phi_t^2>\phi_t^1$ on $U$. Also, since
$\bar p_i=\phi_{t_0}^i(0),$
we have that $0\in U$ is a minimum of $\phi_{t_0}^2-\phi_{t_0}^1$
on $U.$  

By \eqref{eq-invariantchangeorientation},
we can assume that $N_{t_0}^1(p_1)=\gamma'(0)$ and $N_{t_0}^2(p_2)=\gamma'(L).$
In this case, from Lemma \ref{lem-graph}, no  principal curvature
of $F_{t_0}^1(M_1)$ at $p_1$ exceeds  the corresponding one of $F_{t_0}^2(M_2)$ at $p_2.$
Therefore, from the monotonicity property of the Weingarten function $W,$
the following inequality holds:
\begin{equation} \label{eq-3}
W_{F_{t_0}^1}(p_1)\le W_{F_{t_0}^2}(p_2).
\end{equation}

Now, observe that the gradient  of the squared distance function
of $\overbar M$ at the point   $(\bar p_1,\bar p_2)\in\overbar M\times\overbar M$ is the vector
$\nabla {\rm dist}^2(\bar p_1,\bar p_2)=2\,{\rm dist}(\bar p_1,\bar p_2)(-\gamma_{t_0}'(0),\gamma_{t_0}'(L)).$ So,
\begin{equation} \label{eq-4}
\nabla {\rm dist}^2(\bar p_1,\bar p_2)=2\,{\rm dist}(\bar p_1,\bar p_2)(-N_{t_0}^1(p_1),N_{t_0}^2(p_2))\in T_{\bar p_1}\overbar M\times T_{\bar p_2}\overbar M.
\end{equation}

Putting together identities \eqref{eq-1}--\eqref{eq-4}, and considering the fact that
$F_t^1$ and $F_t^2$ are both $W$-flows, we have
\begin{eqnarray}
  D'(t_0) &=& \frac{\partial}{\partial t}{\rm dist}^2(F^1(p_1,t),F^2(p_2,t))|_{t=t_0} \nonumber \\
  &=& \left\langle \nabla {\rm dist}^2(\bar p_1,\bar p_2),
   \left(\frac{\partial F^1}{\partial t}(p_1,t_0), \frac{\partial F^2}{\partial t}(p_2,t_0)\right)\right\rangle_{\overbar M\times\overbar M} \nonumber \\
   &=&    2\,{\rm dist}(\bar p_1,\bar p_2)(-W_{F_{t_0}^1}(p_1)+W_{F_{t_0}^2}(p_2))\ge 0,          \nonumber
\end{eqnarray}
as we wished to prove.
\end{proof}

In the above proof, the hypothesis of $W$ being odd allowed us to choose the orientation
of the hypersurfaces $F_{t_0}^i$ at $p_i,$ $i=1,2,$ in such a way that their unit normals
at these points would coincide with $\gamma'(0)$ and $\gamma'(L).$ In this manner, we could
apply Lemma \ref{lem-graph} and then obtain the fundamental inequality \eqref{eq-3}.
From this, we conclude that  we can drop the assumption on $W$ being odd in the
statement of the avoidance principle, as long as we
have assured that the orientations of $F_t^i$ follow this pattern.

For instance, suppose that $F_t:M^n\rightarrow\overbar M^{n+1},$ $t\in[0,T),$
is a $W$-flow, where $M^n$ is compact and $\overbar M^{n+1}$ is either a space form $\q_\epsilon^{n+1}$  or
a hyperbolic space $\hf.$  Assume that $F_0(M)$ is contained in an open totally convex ball $B_R\subset\overbar M^{n+1},$
whose boundary $\partial B_R$ is a strictly convex geodesic sphere of $\overbar M^{n+1}$
(i.e., $0<R<\pi/2$ for $\overbar M^{n+1}=\s^{n+1}$).
In this setting, considering the parallel flow
$P_t\colon \s^n\rightarrow\overbar M^{n+1}$ with initial data $P_0(\s^n)=\partial B_R$ and inward orientation,
and assuming that $F_t$ is an embedding with the inward orientation for all $t\in [0,T),$ we
have that the normals at the points minimizing the distance between $F_t(M)$ and $P_t(M)$
coincide with the tangent vectors to the minimizing geodesic joining them, as in the above case.
Thus, the avoidance principle holds. In particular, by the results of the preceding section,
$F_t$ has a finite collapsing time which is at most equal to that of $P_t.$

Summarizing, we have the following result.

\begin{corollary} \label{cor-finitetime}
  Let $\overbar M^{n+1}$ be either a space form $\q_\epsilon^{n+1}$  or
  a hyperbolic space $\hf.$
  Given a Weingarten function $W\in C^{\infty}(\Gamma),$ assume that
  $F:M^n\times [0,T)\rightarrow\overbar M^{n+1}$ is a $W$-flow
  of a compact Riemannian manifold $M$ 
  such that $F_0(M)$ is contained in an open totally convex ball $B_R\subset\overbar M^{n+1},$
  whose boundary $\partial B_R$ is a strictly convex geodesic sphere of $\overbar M^{n+1}.$
  Assume further that  one of the following holds:
  \begin{itemize}
    \item $W$ is odd.
    \item $F_t$ is an embedding with the inward orientation for all $t\in [0,T).$
  \end{itemize}
  Under these conditions,  denoting by
  \[P:\s^n\times [0,T_R)\rightarrow\overbar M^{n+1}\]
  the parallel $W$-flow with inward orientation, collapsing time $T_R,$ and initial data $P_0(\s^n)=\partial B_R,$ we have that
  $F_t(M)\cap P_t(\s^n)=\emptyset \, \forall t\in [0,T).$
  Consequently, the inequalities $T\le T_R<\infty$ hold.
\end{corollary}

\begin{proof}[Proof of Theorem \ref{th-embeddednesspreserving}]
Since $F_0$ is an embedding, for any sufficiently small $t>0,$ $F_t$ is also an embedding.
Let us suppose, by contradiction, that there exists a first time $t_0>0$ such that $F_{t_0}$ is not
an embedding. In this way,  $$\Omega:=\{(p,q)\in M\times M\,;\, p\ne q, \,\, F_{t_0}(p)=F_{t_0}(q)\}$$
is a nonempty compact set of $M\times M$ which is disjoint from the diagonal $D$ of $M\times M.$
Thus, there is an open set $U\subset M\times M$ such that  $D\subset U$ and $\Omega$ is disjoint from
the closure of $U$ in $M\times M.$

Now, observing that $V:=(M\times M)-U$ is compact in $M\times M,$ define the function
\[
D(t)=\min_{(p,q)\in V}{\rm dist}^2\,(F_t(p),F_t(q)), \,\,\, t\in[0,t_0].
\]
Since, for a sufficiently small $t,$ $F_t$ is an embedding, for such a $t$ we have
$D(t)>0.$ However, proceeding just as in the proof of Theorem \ref{th-AP}, we conclude
that $D$ is nondecreasing, which contradicts the fact that $D(t_0)=0.$
\end{proof}

\vt
\noindent
{\bf Acknowledgements.}
The author is indebt to F. Fontenele, J. Nazareno Gomes, and M. Marrocos for
helpful conversations during the preparation of this paper.

\end{document}